\documentclass{amsart}
\usepackage{amsmath, amsthm, amssymb, amsfonts, graphicx}
\usepackage[all]{xy}
\usepackage{setspace}
\usepackage{hyperref}
\usepackage{tikz-cd}
\usepackage{amsaddr}

\onehalfspace

\def\Z{\mathbb{Z}}

\def\zplus{\mathbb{Z}^+}

\def\th{^\text{th}}
\def\hom{\text{Hom}}
\def\epsilon{\varepsilon}
\def\phi{\varphi}

\def\diag{\text{Diag}}

\DeclareMathOperator{\im}{im}
\DeclareMathOperator{\ind}{Ind}
\DeclareMathOperator{\coker}{coker}
\let\dim\relax
\DeclareMathOperator{\dim}{Dim}

\def\cfm #1{\text{CFM(#1)}}
\def\rcfm #1{\text{B}(#1)}

\usepackage{xcolor}

\newtheorem{thm}{Theorem}[section]
\newtheorem{prop}[thm]{Proposition}
\newtheorem{cor}[thm]{Corollary}
\newtheorem{lemma}[thm]{Lemma}

\theoremstyle{definition}
\newtheorem{ex}[thm]{Example}
\newtheorem{defn}[thm]{Definition}
\newtheorem{rmk}[thm]{Remark}
\newtheorem{question}[thm]{Question}

\theoremstyle{remark}
\newtheorem*{claim}{Claim}

\begin{document}

\title{Ideal Extensions and Directly Infinite Algebras}


\author{Daniel P. Bossaller}
\email{daniel\_bossaller@baylor.edu}
\address{Department of Mathematics\\
Baylor University\\
1 Bear Place\\
Waco, TX, 76706}

\begin{abstract}
Directly infinite algebras, those algebras, $E$ which have a pair of elements $x$ and $y$ where $1 = xy \neq yx$, are well known to have a sub-algebra isomorphic to $M_\infty(K)$, the set of infinite $\zplus \times \zplus$-indexed matrices which have only finitely many nonzero entries. When this sub-algebra is actually an ideal, we may analyze the algebra in terms of an extension of some algebra $A$ by $M_\infty(K)$, that is, a short exact sequence of $K$-algebras  $0 \to M_\infty(K) \to E \to A \to 0$. The present article characterizes all trivial (split) extensions of $K[x,x^{-1}]$ by $M_\infty(K)$ by examining the extensions as sub-algebras of infinite matrix algebras. Furthermore, we construct an infinite family of pairwise non-isomorphic extensions $\{\mathcal T_i : i \geq 0\}$, all of which can be written as an extension $0 \to M_\infty(K) \to \mathcal T_i \to K[x,x^{-1}] \to 0$.
\end{abstract}

\maketitle

\section{Introduction}
In this article we will assume that all $K$-algebras are unital for some field $K$ of zero characteristic, and that homomorphisms are unital.
The Toeplitz-Jacobson algebra $\mathcal T = \langle x, y : xy = 1\rangle$ was first introduced by Jacobson in \cite{onesidedinverses}. There is a natural decomposition of $\mathcal T$ as an (ideal) extension of $K[x,x^{-1}]$ by $M_\infty(K)$, that is, there is a short exact sequence
\[0 \rightarrow M_\infty(K) \rightarrow \mathcal T \rightarrow K[x,x^{-1}] \rightarrow 0,\] where $M_\infty(K)$ is the (non-unital) algebra of $\zplus \times \zplus$-indexed matrices over $K$ and $K[x, x^{-1}]$ is the algebra of Laurent polynomials.

We dub the ideal $M_\infty(K)$ of $\mathcal T$ a ``faithful" ideal, for when $M_\infty(K)$ is considered as a module over $\mathcal T$, it is faithful as a left and as a right $\mathcal T$-module. It will be shown in Lemma \ref{faithful embed} that having $M_\infty(K)$ as a faithful ideal is equivalent to there being an injective homomorphism from $\mathcal T$ into $\rcfm K$, the algebra of countably infinite matrices where each row and each column has only finitely many elements. This analysis extends to the family of algebras where there exist elements $x$ and $y$ such that $xy = 1$ but $yx \neq 1$; such algebras $E$ are called ``directly infinite." Every directly infinite algebra has an infinite set of matrix units which spans a sub-algebra isomorphic to $M_\infty(K)$. For an exploration of the directly infinite property in the context of von Neumann regular rings, see \cite{goodearl}.

Embeddings similar to that presented in Lemma \ref{faithful embed} have been presented in \cite{gmm} and \cite{nielsen}; however, the embedding in this article specifically allows us to explore the family of faithful extensions (short exact sequences of algebras) $0 \to M_\infty(K) \to E \to A \to 0$ of $A$ by $M_\infty(K)$ through comparison with the extension
\[0 \rightarrow M_\infty(K) \rightarrow \rcfm K \rightarrow \rcfm K /M_\infty(K) \rightarrow 0.\]

Performing such an analysis of extensions through comparison is by no means a new idea. The first analysis of algebra extensions was performed by Hochschild in 1947. In \cite{hochschild}, he embedded extensions within the ``algebra of multiples" of a possibly non-unital algebra $A$, $M(A)$. This multiplier algebra is the smallest algebra which contains $A$ as a faithful ideal. In this construction note that $M(A)$ is only non-trivial when $A$ is non-unital. The author then used techniques of cohomology to put certain equivalence classes of extensions in one to one correspondence with homomorphisms from $C$ into $M(A)/A$. A similar study of these extensions was made by Petrich \cite{petrich} in 1985 which depended on work by Everett from 1942, \cite{everett}. More recent explorations of extensions have been performed by Dorsey and Mesyan in \cite{dorseymesyan} and \cite{IIE}. The present article takes a more constructive view of this problem with an eye towards constructing extensions of $A$ by $M_\infty(K)$.

To that end, Section \ref{pullbacks} is devoted to two perspectives on this problem of comparing the short exact sequence $0 \rightarrow M_\infty \rightarrow E \rightarrow A \rightarrow 0$, with the short exact sequence $0 \rightarrow M_\infty(K) \rightarrow \rcfm K \rightarrow \rcfm K/M_\infty(K) \rightarrow 0$. The first uses embedding of $E$ into $\rcfm K$, to construct a homomorphism $\psi: A \rightarrow \rcfm K/ M_\infty(K)$. The second construction works backward from the homomorphism $\psi$, and uses the pullback in the category of $K$-algebras to make the comparison. This homomorphism $\psi$ from both constructions will be called the ``invariant" of the extension. The techniques from Sections \ref{pullbacks} and \ref{equivalence} are inspired by the work of Busby \cite{busby} who investigated this problem in the the context of $C^*$-algebras. This line of inquiry led to a rich theory of extensions of $C^*$-algebras, most notably in the classification of essentially normal operators by Brown, Douglas, and Fillmore. For a discussion of this classification and some of its consequences, see \cite{davidson}. 

Taking a cue from Busby's work, Section \ref{equivalence} connects an arbitrary extension of $A$ by $M_\infty(K)$ with a homomorphism $\psi: A \rightarrow \rcfm K/M_\infty(K)$, which we call the ``invariant" of the extension. This relationship between the extension and $\psi$ is an equivalence relation. However, this equivalence is too fine in that it separates otherwise isomorphic algebras. To rectify this, a more appropriate notion of equivalence is defined. Theorems \ref{c1 c2 equivalent} and \ref{equivalentinvariants} specifically give characterizations of these two notions of equivalence in terms of this homomorphism $\psi$.

Section \ref{trivial} gives a complete classification of all trivial (split) extensions of $K[x,x^{-1}]$ by $M_\infty(K)$ in terms of the kernel and cokernel of the images of the generators of $K[x,x^{-1}]$ under the invariant (Theorem \ref{trivial extensions}). In order to do this we introduce a family of infinite matrices which are invertible in $\rcfm K/M_\infty(K)$ but not necessarily in $\rcfm K$ and introduce the ``index" of such a matrix. The final section, Section \ref{menagerie} uses the pullback construction of extensions to construct an infinite family of non-isomorphic algebras, each of which has an ideal $M_\infty(K)$ and quotient $K[x,x^{-1}]$.

\section{Background}
In this section, we will formally define and discuss results which will be essential in the remainder of the article. First we define the family of directly infinite algebras, which contains the Toeplitz-Jacobson algebra, $\mathcal T$.

\begin{defn}
An algebra $E$ is called {\it directly infinite} if there exist two elements $x$ and $y$ in $E$ such that $xy = 1$ but $yx \neq 1$.
\end{defn}

Within any directly infinite algebra $E$, there is an infinite set of elements $E_{ij}$ (for $i,j \in \zplus$) with the property that $E_{ij} E_{kl} = \delta_{jk} E_{il}$. We can construct such an set as follows: Since there exist elements $x$ and $y$ in $A$ such that $xy = 1$ but $yx \neq 1$, define $M_E = \{E_{ij} : i,j \in \zplus\}$ by
\[E_{ij} := y^{i-1}(1-yx)x^{j-1}.\]
The linear span of this set $M_E$ is isomorphic to $M_\infty(K)$, the set of infinite $\zplus \times \zplus$ matrices which have only finitely many nonzero elements. When we speak about the matrix units in the family $M_E$, we will use a capital letters, i.e. $E_{ij}$; however, we will denote the matrix units of $M_\infty(K)$ by $e_{ij}$.

The techniques introduced in this article depend on the identification between $M_E$ and $M_\infty(K)$ as well as other infinite matrix algebras. We define those matrix algebras here and note the connection between the matrix algebras and certain endomorphism rings of and infinite dimensional vector space $V$. First, fix a countably-infinite dimensional vector space $V$ with basis $\mathcal B = \{b_i : i \in \zplus\}$. Then $\text{End}(V)$ will denote the $K$-algebra of $K$-linear endomorphisms of $V$. By examining the action of $f \in \text{End}(V)$ on the basis elements, there is a natural identification between $\text{End}(V)$ and $\cfm K$, the $K$-algebra of $\zplus \times \zplus$ matrices in which every column has only finitely many nonzero elements. Of more interest to the present article is the sub-algebra of row and column finite matrices, $\rcfm K$, which has finitely many nonzero entries in each row and column.
To find the endomorphism analogue of $\rcfm K$, recall that $V$ may be written $V = \bigoplus_{n \in \zplus} Kb_i$. Denote by $V_n = \bigoplus_{i = n}^\infty Kb_i$. Then the algebra of row and column finite matrices, $\rcfm K$ is $K$-algebra isomorphic to 
\[\rcfm V = \{f \in \text{End}(V) : \forall n \in \zplus, \exists m \in \zplus \text{ with } f(V_m) \subseteq V_n.\}\]
The set of all $f \in \rcfm V$ such that the image $\im(f)$ is a  finite-dimensional subspace of $V$ will be denoted by $M_\infty(V)$, and this may naturally be associated with the set of infinite matrices $M_\infty(K)$ which have only finitely many nonzero entries.

\begin{rmk}
There are two notational conventions for the algebra of row-and-column finite matrices. Recent articles, see for example \cite{infinitematrices}, have used ``$\text{RCFM}(K)$" to denote this algebra; articles which emphasize the algebraic properties of these matrix algebras tend to use this notation. The second convention can be found in \cite{rcfmexchange} and uses the notation ``$\rcfm K$." This is employed to highlight the similarities between $\rcfm K$ and $\mathcal B(\mathcal H)$, the $C^*$-algebra of bounded linear operators on a (complex) Hilbert space $\mathcal H$. As many of the results in this article are inspired by the similarity between these two structures, we have chosen to follow the latter convention.
\end{rmk}

It is straightforward to see that $M_\infty(K)$ is an ideal of $\rcfm K$. Of note is the fact that this ideal is ``large" within $\rcfm K$ in the following way:
\begin{defn}
	Let $I$ be a two sided ideal of $E$, then $I$ is called {\it faithful} if $aI = \{0\}$ implies that $a = 0$ and $Ib = \{0\}$ implies that $b = 0$. In other words, $I$ is faithful whenever it is faithful as a left $E$-module and as a right $E$-module.
\end{defn}
It should be noted that faithful ideals are similar to, but distinct from, the ``essential ideals" in the theory of $C^*$-algebras.  If $I$ is faithful, then for any ideal $J \subseteq E$, if $I \cap J = \{0\}$ then $J = 0$, thus any faithful ideal is essential. In the theory of $C^*$-algebras, faithful ideals and the set of essential ideals coincide. However, in general, the set of essential ideals properly contains the set of faithful ideals, as can be seen in the algebra $K[x]/(x^2)$ which has an essential, but non-faithful, ideal $(x)$.

Let us return to the Toeplitz-Jacobson algebra. First, as a directly infinite algebra, it has an infinite set of matrix units $M_{\mathcal T}$ which span an ideal isomorphic to $M_\infty(K)$. Furthermore, it can easily be seen that $M_\infty(K)$ is a faithful ideal of $\mathcal T$ by examining the action of $M_\infty(K)$ on the generators of $\mathcal T$.

In \cite{onesidedinverses}, Jacobson gave an embedding of $\mathcal T$ into $\rcfm K$ via shift matrices.  For each integer $i$, $S_i$ will denote the matrix \[S_i = \sum_{j = 1}^\infty e_{i+j,j} \text{ for all } j \text{ such that } j>-i,\] where $e_{ij}$ is the standard set of matrix units. Jacobson defines the embedding: $x \mapsto S_{-1}$ and $y \mapsto S_1$. As a consequence of the definition of $M_{\mathcal T}$, this mapping sends $E_{ij}$ to $e_{ij}$ for all $E_{ij} \in M_{\mathcal T}$. Furthermore, from this embedding it can be seen that the Toeplitz-Jacobson algebra may be written as a short-exact sequence
\[\begin{tikzcd}
	0 \arrow{r} &M_\infty(K) \arrow{r}{i} &\mathcal T \arrow{r}{\pi} &K[x,x^{-1}] \arrow{r} &0.
\end{tikzcd}\]
where $i$ is the natural embedding of $M_\infty(K)$ as an ideal of $\mathcal T$ and $\pi$ is the natural surjection $\mathcal T \to \mathcal T / M_\infty(K)$.
In this article we will call such a short exact sequence an extension of $K[x,x^{-1}]$ by $M_\infty(K)$

\begin{defn}
	Given two algebras $A$ and $C$, we will call the triple $E = (B, \phi, \psi)$ an 
	{\it extension} of $C$ by $A$ when $B$ is an algebra and $\phi: A \to B$ and $\psi: B \to C$ are algebra homomorphisms such that 
	\[\begin{tikzcd}
		0 \arrow{r} &A \arrow{r}{\phi} &B \arrow{r}{\psi} &C \arrow{r} &0 \end{tikzcd}\] is a short-exact sequence.
	When it is assumed that $A$ is a faithful ideal of the algebra $B$, we will call such an extension a {\it faithful} extension.
\end{defn}

To prevent unnecessary notation, when the maps $\phi$ and $\psi$ are understood we will merely state that $B$ is an ``extension" of $C$ by $A$.

The fact that $M_\infty(K)$ is a faithful ideal of $\rcfm K$ is fundamental to the following result by Courtemanche and Dugas, which will be used in Section \ref{trivial}. For simplicity, in this article we denote an inner automorphism $\tau^{-1} v \tau$ by $\widehat \tau (v)$.
\begin{prop}{(\cite{automorphismsoffreemodules}, Lemma 2.2)}
	Let $E$ be a subalgebra of $\text{End}(V)$ such that $E$ contains, as a sub-algebra, the set of finite rank endomorphisms of the vector space $V$, i.e. $\text{Fin}(V) \subseteq E$. If $\alpha$ is an automorphism of $E$, then there exists some invertible $\tau \in \text{End}(V)$ such that $\alpha(v) = \widehat \tau(v)$ for all $v \in V$.
\end{prop}

Translated into the language of infinite matrices which we use in this article, if there is a matrix sub-algebra $E$ of $\cfm K$ which contains the set of all matrices $(a_{ij})_{i,j \in \zplus}$ for which there exists some $n \in \zplus$ such that $a_{ij} = 0$ for all $i > n$ (denote this set by $\text{Fin}(K)$), then every automorphism of $E$ is inner. These matrices are sometimes called ``bounded column-finite" in the literature. The proof depends on the fact that the sub-algebra generated by the set of matrix units is invariant under the automorphism. Hence the result may be rephrased as
\begin{prop}\label{automorphism is inner}
	Let $\alpha$ be an algebra automorphism of $M_\infty(K)$, then there exists some invertible column finite matrix $T$ such that 
	\[\alpha(a) = \widehat T(a) \text{ for all } a \in M_\infty(K).\]
\end{prop}
\noindent Finally a simple argument comparing $T M_\infty(K)$ and $M_\infty(K) T$ shows that $T$ must be row and column finite.

\section{Extensions of Algebras}\label{pullbacks}
Let us take a cue from the Toeplitz-Jacobson algebra and give a method for embedding any directly infinite algebra with faithful ideal $M_\infty(K)$ into $\rcfm K$.
\begin{lemma}\label{faithful embed}
	Suppose that the algebra $E$ has an ideal $M_E$ which is isomorphic to $M_\infty(K)$, then there exists a unique homomorphism $\phi: E \rightarrow \rcfm K$ which makes the following diagram commute.
	\[\begin{tikzcd}
		&\rcfm K\\
		M_E \arrow{ur}{\iota} \arrow{r}[swap]{i} &E \arrow[dashed]{u}[swap]{\phi}
	\end{tikzcd}\]
	where $\iota$ is the natural embedding of $M_E$ into $\rcfm K$ and $i$ is the embedding of $M_E$ as an ideal of $E$. Furthermore $\phi$ is injective if, and only if, $M_E$ is a faithful ideal of $E$.
\end{lemma}
\begin{proof}
	Define $\phi$ by $\phi(a) = (a_{ij})$ where $(a_{ij})$ is the matrix whose entries are defined $a_{ij} := E_{ii}aE_{jj}$ (note also that $a_{ij} \in K$). Because any decomposition of an algebra by a complete set of orthogonal idempotents is a direct sum, it is simple to show that $(a_{ij}) \in \rcfm K$ for any $a \in E$. Commutativity of the diagram is similarly straightforward. So all that remains to show is uniqueness. Suppose that there exists some $\psi$ which also completes the commutative diagram. Consider for all choices of $i, j \in \zplus$, the $(i,j)\th$ entry of $\psi(a)$. Then $e_{ii} \psi(a) e_{jj} = \psi(E_{ii}) \psi(a) \psi(E_{jj}) = \psi(E_{ii}aE_{jj}) = E_{ii} a E_{jj}$. Because the $(i,j)$ entry of $\psi$ is equal to the $(i,j)$ entry of $\phi$, it follows that $\psi = \phi$. This completes the proof of the first statement.
	
	Now suppose that $\phi$ is injective, but that $M_E$ is not a faithful ideal. Then there is some nonzero $a \in E$ such that $aM_A = 0$ or $M_Aa = 0$. In this case $\phi(a) = (E_{ii} a E_{jj}) = 0$, contradicting the injectivity of $\phi$. Now suppose that $M_E$ is faithful. If $\phi(a)$ were not injective, then there would be some nonzero $a \in E$ such that $E_{ii} a E_{jj} = 0$ for all $i,j \in \zplus$. In particular this means that $E_{ii} a = 0$ and $a E_{jj} = 0$. By noting that $E_{ij} = E_{ii} E_{ij} E_{jj}$, one has that $a E_{ij} = E_{ij} a = 0$. Thus $a$ is a nonzero element of $E$ such that $a M_A = 0$ and $M_E a = 0$, contradicting that $M_E$ is faithful.
\end{proof}

\begin{rmk}
	Two things should be mentioned about this result. First, in the case of a faithful ideal $M_E$ of $E$, $\phi$ is completely determined by the embedding $\iota: M_E \rightarrow \rcfm K$ due to the construction of $(a_{ij})$. Second the assumption of a faithful $M_E$ is necessary. One need only consider $E = B \oplus M_\infty(K)$, the co-product of $B$ with $M_\infty(K)$ for some $K$-algebra $B$. The map $\phi$ defined above merely becomes a projection onto the second coordinate, which is not injective. 
\end{rmk}

While many of the following results do not require the assumption that $M_E$ be a faithful ideal of $E$ (most notably Theorem \ref{c1 c2 equivalent}), we will, nevertheless make that assumption throughout this article. This assumption admits two simplifications, first we may think about $E$ as a matrix subalgebra of $\rcfm K$, and second, we will identify $M_E$ and $M_\infty(K)$. This second simplifications avoids the hassle of specifically defining the isomorphism $f:M_E \rightarrow M_\infty$ in each of the calculations.

With this convention, we study the extension 
\[\begin{tikzcd}
	0 \arrow{r} &M_\infty(K) \arrow{r} &E \arrow {r} &A \arrow{r} &0
\end{tikzcd}\]by comparing it with the extension 
\[\begin{tikzcd}
	0 \arrow{r} &M_\infty(K) \arrow{r} &\rcfm K \arrow {r} &Q(K) \arrow{r} &0
\end{tikzcd}\] with $Q(K) := \rcfm K / M_\infty(K)$. This extension is ``largest" extension in the sense that $\rcfm K$ is algebra of multipliers of $M_\infty(K)$, in other words, $\rcfm K$ is the largest algebra which contains $M_\infty(K)$ as a faithful ideal, (\cite{regmul} Proposition 1.1).

Two equivalent constructions will be given for this comparison; the first extends the map defined in Lemma \ref{faithful embed}. The second will work backward from a homomorphism from $A$ to $Q(K)$.
\subsection{Construction 1}

Define the quotient map $\pi: \rcfm K \rightarrow Q(K)$, its restriction map $\pi|_E: E \rightarrow A$, and construct the following commutative diagram featuring the to-be-defined map $\psi: A \rightarrow Q(K)$.

\begin{equation}\label{const1}\begin{tikzcd}
		0 \arrow{r} &M_\infty(K) \arrow[hook]{r}{\iota} \arrow[equals]{d} &\rcfm K \arrow{r}{\pi} &Q(K) \arrow{r} &0\\
		0 \arrow{r} &M_\infty(K) \arrow[hook]{r}[swap]{i} &E \arrow{u}[swap]{\phi} \arrow{r}[swap]{\pi|_E}&A \arrow[dashed]{u}[swap]{\psi}\arrow{r} &0
\end{tikzcd}\end{equation}
Let $\psi: A \rightarrow Q(K)$ be the image of $\phi$ on cosets, that is $a + M_\infty(K) \mapsto \phi(a) + M_\infty(K)$ for any $a \in E$. Certainly this is a well-defined homomorphism, and it is straightforward to see that this produces a commutative diagram.

\begin{lemma}\label{faithfulideal}
	$\psi$ is injective if and only if $M_\infty(K)$ is a faithful ideal of $E$.
\end{lemma}
\begin{proof}
	Suppose that $M_\infty(K)$ is a faithful ideal of $E$, note that Lemma \ref{faithful embed} shows that $\phi$ is an embedding. Say that that there exists some $a \in A$ such that $\psi(a) = 0$. Thus $\phi(a) + M_\infty(K) = 0$, which implies that $0 = \phi(a) - m = \phi(a - m)$ for all $m \in M_\infty(K)$. Because $\phi$ is injective, $a = m$, which further implies that $a = 0$ in $A \simeq E/M_\infty(K)$. So $\psi$ is injective.
	
	If $\psi$ is injective, a diagram chase assures that $\phi$ is injective. Thus by Lemma \ref{faithful embed}, $M_\infty(K)$ must be an faithful ideal of $E$.
\end{proof}

\begin{defn}
	This homomorphism $\psi$ will be called the {\it invariant} of the extension $E$.
\end{defn}

\subsection{Construction 2}

This construction will start with the invariant $\psi$ and work backward. A key tool in this construction will be the pullback in the category of $K$-algebras, which we recall here.

Given algebras $A$, $B$, and $C$ and homomorphisms $f: A \to C$ and $g: B \to C$,  the pullback of $C$ along the homomorphisms $f$ and $g$ is constructed
\[P  = A \oplus_C B= \{(a,b) \in A \oplus B : f(a) = g(b)\}.\]  Furthermore, recall that the pullback is unique up to isomorphism.

Suppose that there is a homomorphism $\psi: A \rightarrow Q(A)$ such that we have the following diagram (dashed arrows and $\bullet$'s indicate to-be-filled in homomorphisms and algebras, respectively):
\[\begin{tikzcd}
	0 \arrow{r} &M_\infty(K) \arrow{r}{i}& \rcfm K \arrow{r}{\pi}&Q(K) \arrow{r} &0\\
	0 \arrow[dashed]{r} &\bullet \arrow[dashed]{r} \arrow[dashed]{u} &\bullet\arrow[dashed]{r} \arrow[dashed]{u} &A\arrow{u}{\psi} \arrow[dashed]{r} &0
\end{tikzcd}\]
Define the pullback $P$ of $Q(K)$ along the homomorphisms $\pi$ and $\psi$ with maps $\alpha: P \rightarrow \rcfm K$ and $\beta: P \rightarrow A$ defined as the projections on the first and second coordinates respectively. Thus $P = \rcfm K \oplus_{Q(K)} A$. Finally define a map $f: M_\infty(K) \rightarrow P$ by $m \mapsto (m, 0)$. The range of $f$ is certainly in $P$ since $\pi(m) = 0 = \psi(0)$. The only things left to check are the exactness of the bottom row and commutativity of the left square of the following diagram, whose proofs are straightforward.

\begin{equation}\label{const2}\begin{tikzcd}
		0 \arrow{r} &M_\infty(K) \arrow[equals]{d} \arrow{r}{i}& \rcfm K \arrow{r}{\pi}&Q(K) \arrow{r} &0\\
		0 \arrow{r} &M_\infty(K) \arrow{r}[swap]{f} &P \arrow{u}{\alpha} \arrow{r}[swap]{\beta} &A\arrow{u}{\psi} \arrow{r} &0
\end{tikzcd}\end{equation}

\begin{rmk}
	Note that this $\alpha: P \rightarrow \rcfm K$ satisfies the conditions of Lemma \ref{faithful embed} and is thus the unique homomorphism guaranteed in the argument. Thus Construction 1 assures that $\psi$ is an invariant of the extension $(P, f, \beta)$ of $A$ by $M_\infty(K)$.
\end{rmk}

\section{Equivalence of Extensions}\label{equivalence}
These two constructions are connected in the following way.

\begin{prop}
	Suppose that $(E, i, \pi)$ is an extension of $A$ by $M_\infty(K)$ with invariant $\psi: A \rightarrow Q(K)$ as in Equation \ref{const1}. Then there is an isomorphism $\Phi$ between $E$ and the pullback $P$ of $Q(K)$ along $\psi$ and $\pi$ making the following diagram commute.
	\[\begin{tikzcd}
		0 \arrow{r} &M_\infty(K) \arrow[equals]{d} \arrow{r}{i} &E \arrow{d}{\Phi} \arrow{r}{\pi|_E} &A \arrow{r} \arrow[equals]{d} &0\\
		0 \arrow{r} &M_\infty(K) \arrow{r}[swap]{f} &P \arrow{r}[swap]{\beta} &A \arrow{r} &0
	\end{tikzcd}\]
\end{prop}
\begin{proof}
	Define $\Phi: E \rightarrow P$ by $\Phi(a) = (\phi(a), \pi|_E(a))$; the map $\phi$ is as in Equation \ref{const1}. Note that $\pi(\phi(a)) = \psi (\pi|_E(a))$, so $(\phi(a), \pi|_E(a)) \in P$.
	
	Now we claim that $\Phi$ is a bijection. To show that it is injective, suppose that $0 = \Phi(a) = (\phi(a), \pi|_E(a))$. Since $\pi|_E$ is the surjection onto $E/M_\infty(K)$, $a \in M_\infty(K)$. But, since $\phi$ acts as the identity on $M_\infty(K)$, we must have that $0 = \phi(a) = a$.
	
	Now suppose that there is some $(m, a) \in P$. Since $a \in A$ and $\pi|_E$ is surjective, there is some $x \in E$ such that $\pi_E(x) = a$. Now let us construct the element $\phi(x) - m \in \rcfm K$. 
	\begin{equation*}
		\begin{split}
			\pi(\phi(x) - m) &= \pi(\phi(x)- \pi(m))\\
			&= \psi(\pi|_E(x))- \pi(m)\\
			&= \psi(a)- \pi(m) = 0.
		\end{split}
	\end{equation*}
	The last equality is due to the fact that in $P$, $\psi(a) = \pi(m)$. Thus there is some $k \in M_\infty(K)$ such that $\phi(x) = m + k$. Define $\hat x = x - k \in E$. Then $\phi(\hat x) = \phi(x - k) = \phi(x) - k = m$. Furthermore, $\pi|_E(\hat x) = \pi|_E(x) = a$, and $\Phi(\hat x) = (m,a)$.
	
	Commutativity of the diagram follows from a simple calculation, which completes the proof.
\end{proof}

This motivates the following definition

\begin{defn}
	Let $(E_1, i_1 \pi_1)$ and $(E_2, i_2, \pi_2)$ be two extensions of $A$ by $M_\infty(K)$, then the two extensions are {\it strongly equivalent} if there is some isomorphism $\Phi: E_1 \rightarrow E_2$ which makes the following diagram commute
	\[\begin{tikzcd}
		0 \arrow{r} &M_\infty(K) \arrow[equals]{d} \arrow{r}{i_1} &E_1 \arrow{d}{\Phi} \arrow{r}{\pi_1} &A \arrow{r} \arrow[equals]{d} &0\\
		0 \arrow{r} &M_\infty(K) \arrow{r}[swap]{i_2} &E_2 \arrow{r}[swap]{\pi_2} &A \arrow{r} &0
	\end{tikzcd}\]
	
	It is straightforward to see that strong equivalence is an equivalence relation, thus denote by $[E_\psi]$ the strong equivalence class of the pullback algebra $E_\psi$ of $Q(K)$ along $\psi$ and $\pi$.
\end{defn}

The following result, which was first presented in a more general form by Hochschild \cite{hochschild}, is a direct consequence of the uniqueness of the pullback and the equivalence of the two constructions.

\begin{thm}\label{c1 c2 equivalent}
	There is a one-to-one correspondence between the set of algebra homomorphisms $\hom(A, Q(K))$ and the strong equivalence classes of extensions of $A$ by $M_\infty(K)$
\end{thm}

\begin{rmk}
	While useful in the construction of algebras, the notion of strong equivalence is too fine for our purposes. Consider the pullback of $Q(K)$ along the natural surjection $\pi$ and invariant $\rho$ defined by  
	\[\rho(x) = T_{-1} := \sum_{i=1}^\infty \frac{1}{i+1}e_{i,i+1} \text{ and } \rho(x^{-1}) = T_{1} := \sum_{j =1}^\infty (j+1) e_{j+1,j}\]
Two things are of note. First, our assumption that $K$ is a field of characteristic $0$ is necessary for this to be well-defined. Second, $T_{-1}T_1 = I_\infty$, but $T_1T_{-1} = I_\infty - e_{11}$. Since these two differ in only one entry, $\psi$ is an embedding of $K[x,x^{-1}]$ into $Q(K)$, hence this is a faithful extension of $K[x,x^{-1}]$ by $M_\infty(K)$. Furthermore, the mapping $T_{-1} \mapsto x$ and $T_{1} \mapsto y$ shows that this extension is isomorphic to $\mathcal T$. If $\psi_1$ is the invariant of the Jacobson embedding of $\mathcal T$, namely $\psi_1(x) = S_{-1}$ and $\psi_1(y) = S_1$,  then it is obvious that $\rho(x) - \psi_1(x) \notin M_\infty(K)$. Hence the two extensions occupy separate strong equivalence classes, which is not desirable for two otherwise isomorphic matrix algebras. 
\end{rmk}
\noindent To remedy this, we introduce a weaker notion of equivalence for extensions.

\begin{defn}
	Two extensions $E_1$ and $E_2$ of $A$ by $M_\infty(K)$ are said to be {\it equivalent} if there is an isomorphism $\Phi: E_1 \rightarrow E_2$ which restricts to an automorphism of $M_\infty(K)$ and makes the following diagram commute.
	\[\begin{tikzcd}
		0 \arrow{r} &M_\infty(K) \arrow{r}{i_1} \arrow{d}[swap]{\Phi|_{M_\infty(K)}} &E_1 \arrow{r}{\pi|_{E_1}} \arrow{d}{\Phi} &A \arrow[equals]{d}\arrow{r} &0\\
		0 \arrow{r} &M_\infty(K) \arrow{r}{i_2} &E_2 \arrow{r}{\pi|_{E_2}} &A \arrow{r} &0
	\end{tikzcd}.\]
\end{defn}
\begin{rmk}\label{equivalent embed}
	It is clear to see that any two strongly equivalent extensions are equivalent and this relation is also an equivalence relation. Furthermore, the two extensions $E_1$ and $E_2$ given by invariants $\psi_1$ and $\rho$, respectively, in the previous remark are equivalent. Note that we may define the natural isomorphism $\Phi: E_1 \to E_2$ by $\Phi(S_i) = T_i$ for $i \in \{-1,1\}$ and extending linearly. This then induces an automorphism of the ideal $M_\infty(K)$ by the mapping
	\[\Phi|_{M_\infty(K)}(e_{ij}) = \frac{i!}{j!}e_{ij},\]
	hence the equivalence.
\end{rmk}

The following result re-frames this equivalence condition in terms of the invariants of the respective extensions.

\begin{thm}\label{equivalentinvariants}
	Let $\psi_1, \psi_2: A \rightarrow Q(K)$ be two embeddings of $A$ into $Q(K)$, leading to faithful extensions $E_1$ and $E_2$. $E_1$ and $E_2$ are equivalent if and only if there is some invertible matrix $U \in \rcfm K$ such that $\widehat{\pi(U)}(\psi_1) = \psi_2$.
\end{thm}
\begin{proof}
	Suppose that there exists an isomorphism $\Phi: E_1 \rightarrow E_2$ which restricts to an automorphism on $M_\infty(K)$. Define $f_{ij} = \Phi(e_{ij})$ for all $i, j \in \zplus$, then by Theorem \ref{automorphism is inner}, there exists some invertible matrix $U \in \rcfm K$ such that $\Phi(m) = \widehat{U}(m)$ for all $m \in M_\infty(K)$ and define $f_{ii} = \Phi(e_{ii})$. Let $a \in E_1$ and then calculate
	\begin{equation*}
		\begin{split}
			f_{ii} \Phi(a) f_{jj} &= \Phi(e_{ii}) \Phi(a) \Phi(e_{jj}) = \Phi(e_{ii} a e_{jj})\\
			&= U^{-1}(e_{ii} a e_{jj})U = f_{ii} \widehat{U}(a) f_{jj}
		\end{split}
	\end{equation*}
	
	Because $f_{ii} \Phi(a) f_{jj} = f_{ii} \widehat{U}(a) f_{jj}$ for all $a \in E_1$ and all $i,j \in \zplus$ and $M_\infty(K)$ is an faithful ideal of $A_2$, it is evident from the construction of the embedding of $A_2$ into $\rcfm K$ given in Lemma \ref{faithful embed} that $\Phi(a) = \widehat{U}(a)$. Now suppose that $(m, b) \in E_1 = \rcfm K \oplus_{Q(K)} A$ and that $(m', b') = \Phi((m,b)) \in E_2$. Thus $\psi_2(b') = \pi(m') = \pi(U^{-1} m  U) = \pi(U)^{-1} \pi(m) \pi(U) = \widehat{\pi(U)}(\psi_1(b))$, and $\rho = \widehat{\pi(U)}(\psi_1)$.
	
	Now suppose that there exists some invertible matrix $U \in \rcfm K$ such that $\widehat{\pi(U )}(\psi_1) = \psi_2$. We claim that $\Phi := \widehat U$ is the desired isomorphism which restricts to an automorphism of $M_\infty(K)$. Due to the construction of Lemma \ref{faithful embed}, $E_1$ and $E_2$ may be thought of as infinite matrix sub-algebras of $\rcfm K$ with faithful ideals isomorphic to $M_\infty(K) \subseteq \rcfm K$. Conjugation by a row and column finite matrix is an automorphism of $M_\infty(K)$, thus $\Phi = \widehat{U}$ restricts to an automorphism of $M_\infty(K)$.
	
	It remains to show that $\Phi := \widehat U$ is surjective, since it is certainly an injective algebra homomorphism. Suppose $x \in E_2$, then $\pi(x) \in Q(K)$, and by construction $\pi(x) = \psi_2(a)$ for some $a \in A$. Then, for some $a' \in A$, $\pi(x) = \widehat{\pi(U)}(\psi_1(a')) = \pi(\widehat U (x'))$ for some $x' \in E_1$. Thus $x = \widehat U (x' + m)$ for some $m \in M_\infty(K)$. Because $M_\infty(K)$ is an ideal of $E_1$, it is clear that $x' + m \in E_1$, and $\Phi$ is surjective.
\end{proof}

\section{Trivial Extensions of \texorpdfstring{$K[x,x^{-1}]$}{the Laurent polynomials} by \texorpdfstring{$M_\infty(K)$}{M(K)}}\label{trivial}

As an extension is a short exact sequence of algebras, a logical first step in the classification of an extension is examining when the short exact sequece splits.

\begin{defn}
	Let $(E, i, \pi)$ be a faithful extension of $A$ by $M_\infty(K)$, $E$ is called {\it trivial} if there is an algebra homomorphism $\rho: A \to E$ such that $\pi \circ \rho = \text{id}_A$. This then implies that the extension is trivial if and only if there is an injective homomorphism $\sigma: A \rightarrow \rcfm K$ such that $\pi \circ \sigma = \psi$.
\end{defn}

In this section we will give a complete characterization for trivial extensions of $K[x,x^{-1}]$ by $M_\infty(K)$. As a corollary, we will provide an alternate proof of Theorem 2 from \cite{structurelpapoly}, showing that the Toeplitz-Jacobson algebra is not a trivial extension. 

Our starting point will be an observation about the invariant of a faithful extension $E$ of $K[x,x^{-1}]$ by $M_\infty(K)$. Consider $\psi(x)$ and $\psi(x^{-1})$. In our examples thus far, the images of $x$ and $x^{-1}$ have been invertible in $Q(K)$. As a matter of fact, in order to have a faithful extension of $K[x,x^{-1}]$, this must always be the case. However, none of the pre-images $\pi^{-1}(\psi(x)) \in \rcfm K$ in our examples have been invertible. Such matrices will play a key role in our classification of trivial extensions.

\begin{defn}
	A matrix $A \in \rcfm K$ is called {\it algebraically Fredholm} if the image of $A$ under the natural surjection $\pi: \rcfm K \rightarrow Q(K)$ is invertible in $Q(K)$. In the remainder of this article, we will refer to such matrices simply as ``Fredholm matrices."
\end{defn}

In other words, $A$ is Fredholm if and only if  there exists $A_1, A_2 \in \rcfm K$ and $S_1, S_2 \in M_\infty$ such that $AA_1 = I_\infty - S_1$ and $A_2 A = I_\infty - S_2$. A consequence of this characterization: since $A_1$ and $A_2$ differ only by some element $R \in M_\infty(K)$, we can select some $A_0$ which functions as both a left and a right Fredholm inverse. That is, there exists $A_0$, $R_1$ and $R_2$ such that $A A_0 = I_\infty - R_1$ and $A_0 A = I_\infty - R_2$. Note that this also implies that Fredholm inverses are unique up to perturbation by some element of $M_\infty(K)$; when we say ``the" Fredholm inverse of a matrix, it is understood within this context.

\begin{prop}
	The family of Fredholm matrices is closed under multiplication.
\end{prop}
\begin{proof}
	Say that $A$ and $B$ are matrices such that $\overline A$ and $\overline B$ are invertible in $Q(K)$. Suppose $A$ has Fredholm inverse $A_0$ and $B$ has Fredholm inverse $B_0$. Then consider $\overline{AB} = \bar{A} \bar{B}$ which is clearly invertible in $Q(K)$ with Fredholm inverse $\bar{B_0}\bar{A_0}$.
\end{proof}

In the theory of Banach algebras, Fredholm operators are defined as those operators $T$ with closed range such that $\dim(\ker(T))$ and $\dim(\ker(T'))$ are finite (where $T'$ is the Hilbert space adjoint of $T$). The following result establishes that if a matrix is Fredholm, then $\dim(\ker(T))$ and $\dim(V/TV)$ are finite. Furthermore, using this fact we will introduce the ``index" of a matrix which will function as a measurement for how far a given Fredholm matrix is from being invertible. For any matrix $A \in \rcfm K$, will use the notation $\im(A)$, $\ker(A)$, and $\coker(A)$ for the image, kernel, and cokernel (respectively) of the endomorphism $\mathcal L_{A}$, where $\mathcal L_A$ denotes the linear transformation $x \mapsto Ax$.

\begin{lemma}\label{fredholm dimension}
	If $A \in \rcfm K$ is a Fredholm matrix then $\ker(A)$ and $\coker(A)$ are finite dimensional subspaces of $V$.
\end{lemma}
\begin{proof}If $A$ is Fredholm, then there exists $A_0\in \rcfm K$ and $R, S \in M_\infty(K)$ be matrices such that $A_0A = I_\infty - R$ and $AA_0 = I_\infty - S$.
	
	To show that $\ker(A)$ is finite dimensional, suppose that there is an infinite, linearly independent set of elements in $\ker(A)$, $\{b_i : i \in \zplus\}$. Then $Ab_i = 0$ for each $i \in \zplus$. Construct a matrix $B = (b_1 \; |\; b_2\; |\; \cdots\; )$. By construction $AB = 0$.  Then
	\[0 = A_0(AB) = (A_0A)B = B - RB.\] Thus $B = RB$; this is a contradiction; thus for a Fredholm matrix $A$, $\ker(A)$ must be finite dimensional. 
	
	For the other claim, first note that $\im(AA_0) \subseteq \im(A)$; thus $\im(I_\infty - S) \subseteq \im(A)$ which means that $\coker(A) \subseteq \coker(I_\infty - S)$. Thus 
	\[\dim(\coker(A)) \leq \dim(\coker(I_\infty - S)).\]
	We claim that $\dim(\coker(I_\infty - S))$ also finite. Note that
	\[\coker(I_\infty - S) = V / \im(I_\infty - S) \simeq \ker(I_\infty-S)  \subseteq \im(S).\] Since the dimension of $\im(S)$ is finite, the dimension of the cokernel must be finite also, which proves the claim.
\end{proof}

\begin{defn}
	Let $A$ be a Fredholm matrix, we define the {\it index} of $A$ to be 
	\[\ind(A) = \dim(\ker(A)) - \dim(\coker(A)).\]
\end{defn}

Note that the shift matrices $S_i$ and the $T_i$ from Remark \ref{equivalent embed} are Fredholm and one may calculate their indices:
\[\ind(S_i) = -i, \text{ and } T_{\pm 1} = \mp 1\]
for every $i \in \Z$. The proof of the following proposition is adapted from the theory of Fredholm operators in functional analysis. See, for example, \cite{schechter} Chapter 5.

\begin{prop}\label{fredholm properties}
	Let $A$ and $B$ be Fredholm matrices and let $T \in M_\infty(K)$, and let $A_0$, $R$, and $S$ be as in the proof of Lemma \ref{fredholm dimension}.
	\begin{enumerate}
		\item $\ind(AB) = \ind(A) + \ind(B)$.
		\item $\ind(A_0) = - \ind(A)$. 
		\item $A+T$ is Fredholm, and $\ind(A + T) = \ind(A)$.
	\end{enumerate}
\end{prop}
\begin{proof}
	To prove (1), we divide $V$ up into four subspaces $V_1$, $V_2$, $V_3$, and $V_4$.
	\[\begin{array}{r c l}
		V_1 &= &\ker(A) \cap \im(B),\\
		\im(B) &= &V_1 \oplus V_2,\\
		\ker(A) &= &V_1 \oplus V_3 \text{, and from here we get}\\
		V &= &\im(B) \oplus V_3 \oplus V_4.
	\end{array}\]
	
	Note that by Lemma \ref{fredholm dimension}, $V_1$ and  $V_3$ are subspaces of $\ker(A)$, and $V_4$ is isomorphic to a subspace of $\coker(B)$. Hence, by assumption, $V_1$, $V_3$, and $V_4$ are all finite dimensional. Thus we may define $d_i = \dim(V_i)$ for $i \in \{1, 3, 4\}$. In addition, one can find two more subspaces $W, X \subseteq V$ by writing $\ker(AB) = \ker(B) \oplus W$ and $\im(A) = \im(AB) \oplus X$. Since $W \subseteq \ker(AB)$ and $X$ is isomorphic to some subspace of $\coker(AB)$, both $W$ and $X$ are finite dimensional. Note that $W$ is the subspace of all vectors $v$ such that $v \in \im(B)$ but $v \in \ker(A)$, so $\dim(W) = d_1$. Also note that $\im(A) = \mathcal{L}_A(V) = \mathcal{L}_A(\im(B) \oplus V_3 \oplus V_4) = \im(AB) \oplus \mathcal{L}_A(V_4)$. Since $\ker(A) = V_1 \oplus V_3$, $\mathcal{L}_A$ must be a one-to-one linear transformation from $V_4$ to $W$ which implies that $V_4$ and $X$ must have the same dimension, $\dim(X) = d_4$.
	
	Collecting our work from the previous paragraphs, we have that 
	\[\begin{array}{r c l}
		\dim(\ker(AB)) &= &\dim(\ker(B)) + d_1\\
		\dim(\coker(AB)) &= &\dim(\coker(A)) + d_4\\
		\dim(\ker(A)) &= &d_1 + d_3\\
		\dim(\coker(B) &= &d_3 + d_4
	\end{array}.\]
	So we calculate $\ind(AB) = \dim(\ker(B)) + d_1 - \dim(\coker(A)) - d_4$. On the other hand, $\ind(A) + \ind(B) = \dim(\ker(A)) - \dim(\coker(A)) + \dim(\ker(B)) - \dim(\coker(B)) = d_1 + d_3 - \dim(\coker(A)) + \dim(\ker(B)) - d_3 + d_4$, which gives the desired equality.
	
	The proof of (2) follows from the fact that $\ker(I_\infty - R) = \{v \in V : v - R v = 0\} = \{v \in V : v = R v\} \simeq \coker(I_\infty - R)$. Because those subspaces have finite dimension, we calculate 
	\[0 = \ind(I_\infty - S) = \ind(A_0 A) = \ind(A_0) + \ind(A).\]
	
	To show that (3) holds, define $R' = (R -A_0T)$ and $S' = (S - T A_0)$, and note \[\begin{array}{r c l} A_0(A + T) &= &I_\infty - R + A_0 T = I_\infty - R'\\
		(A+T)A_0 &= &I_\infty - S + TA_0 = I_\infty - S'
	\end{array}.\] Thus $A+T$ is Fredholm. Finally,
	\[\ind(A_0) + \ind(A+T) = \ind(A_0(A+T)) = \ind(I_\infty - R') = 0.\] Since $\ind(A_0) = - \ind(A)$, we have that $\ind(A) = \ind(A+T)$.
\end{proof}

As a corollary we are able to show the following properties of the index.

\begin{cor}\label{conjugation}
	Let $U$ be an invertible matrix in $\rcfm K$, and let $A$ and $B$ be Fredholm matrices. Furthermore, let $B_0$ be a Fredholm inverse of $B$. Then the following properties hold.
	\begin{enumerate}
		\item $\ind(U) = 0$
		\item $\ind(A) = \ind(U^{-1} A U)$.
		\item $\ind(A) = \ind(B_0 A B)$.
		\item $\ind(A) = 0$ if and only if $A = V + T$ for some invertible $V \in \rcfm K$, and $T \in M_\infty(K)$
	\end{enumerate}
\end{cor}
\begin{proof}
	The first claim follows from the fact that if $U$ is invertible, its kernel and cokernel are trivial. The second follows from the first, and the third follows from Proposition \ref{fredholm properties}. 
	
	So all that is left to show is (4). Sufficiency follows directly from (1) and from Proposition \ref{fredholm properties}. This proof will use the definition of  $\rcfm V$ from the introduction:
	\[\rcfm V = \{f \in \text{End}(V) : \forall n \in \zplus, \exists m \in \zplus \text{ with } f(V_m) \subseteq V_n\}.\] Suppose that $f$ is a Fredholm endomorphism of index 0, that is, $f$ is an invertible linear transformation in $\rcfm V/M_\infty(V)$ for which $\dim(\ker(f)) = \dim(\im(f))$, and both are finite. Decompose the vector space $V$ two different ways: 
	\[V \simeq \ker(f) \oplus V/\ker(f) \simeq \coker(f) \oplus \im(f)\]
	Finally define isomorphisms $\bar f: V/\ker(f) \rightarrow \im(f)$ and $\phi: \ker(f) \rightarrow \coker(f)$.
	
	Then define a map $\phi \oplus \bar f: V \rightarrow V$. As it is the direct sum of two injective homomorphisms with image $\coker(f) \oplus \im(f) = V$, this is an automorphism of $V$. Now let $v \in V$ and decompose it as $v = v_0 + v_1$ where $v_0 \in \ker(f)$ and $v_1 \in V/\ker(f)$. Then consider
	\[\left[f - (\phi \oplus \bar f)\right](v) = f(v_0 + v_1) - \phi(v_0) -\bar f(v_1) = \phi(v_0) + (f - \bar f)(v_1) = \phi(v_0).\]
Thus these two homomorphisms differ only by some element from $\coker(\phi)$ which is finite dimensional. 

All that is left to show is that $\phi \oplus \bar f \in \rcfm V$. Since $f \in \rcfm V$, for any $n \in \zplus$ there is some $m \in \zplus$ such that $f(V_m) \subseteq V_n$. Then that same choice of $m$ will assure that $f\left(V_m/\ker(f)|_{V_m}\right) \subseteq V_n$. Thus $\bar f \in \rcfm V$. Furthermore, as the $V_i$ form a countable descending sequence of subspaces, this $m$ may be chosen large enough so that $V_m \cap \ker(f) = \{0\}$. Thus $\phi \oplus \bar f \in \rcfm K$, and $f$ is the sum of an invertible element of $\rcfm V$ and an element of $M_\infty(V)$. In terms of matrices, this implies that every Fredholm matrix of zero index is the sum of an invertible and a finite rank matrix.
\end{proof}

We are now in the position to begin the classification of trivial extensions of $K[x,x^{-1}]$ by $M_\infty(K)$.

\begin{lemma}
If $E$ is an extension of $K[x,x^{-1}]$ by $M_\infty(K)$ with invariant $\psi$ such that $\ind(\psi(x)) = 0 = \ind(\psi(x^{-1}))$, then $E$ is a trivial extension.
\end{lemma}
\begin{proof}
	We need to show that there is an injective homomorphism\linebreak $\sigma: K[x,x^{-1}] \to \rcfm K$ such that $\pi\circ \sigma = \psi$. Write $\psi(x) = T + m$ and $\psi(x^{-1}) = T' + m'$ for invertible matrices $T$ and $T'$ and $m,m' \in M_\infty(K)$. Also note that $\psi(x) \psi(x^{-1}) = \psi(1) = I_\infty + n$ for some $n \in M_\infty(K)$. Thus
	\[(T+m)(T'+ m') = TT' + M = I_\infty + n\] for $M = mT' + Tm' + mm' \in M_\infty(K)$. Now let $T^{-1}$ be the inverse of $T$, and calculate $T(T' - T^{-1}) = n - M$, which then shows $T' - T^{-1} = T^{-1}(n-M) \in M_\infty(K)$. Thus $\psi(x^{-1}) = T^{-1} + m''$.
	
	Define $\sigma: K[x,x^{-1}] \to \rcfm K$ by $x \mapsto T$ and $x^{-1} \mapsto T^{-1}$ and extend linearly. Thus $E$ is a trivial extension.
\end{proof}

We thus have part of the following theorem.

\begin{thm}\label{trivial extensions}
	An extension $E$ of $K[x,x^{-1}]$ by $M_\infty(K)$ is trivial if and only if $\ind(\psi(x)) = 0 $.
\end{thm}
\begin{proof}
	We must only show necessity, to that end, suppose that $E$ is trivial. Then there is some homomorphism $\sigma$ as above such that $\sigma(x)$ and $\sigma(x^{-1})$ are units in $\rcfm K$. Thus $\ind(\pi \circ \sigma(x)) = \ind(\psi(x)) = 0$.
\end{proof}

As a corollary we have established the following result by Alahmedi et. al. 
\begin{cor}{(\cite{structurelpapoly}, Theorem 2)}
	The Toeplitz-Jacobson algebra is not a trivial extension.
\end{cor}
\begin{proof}
	Consider the Jacobson embedding of the Toeplitz-Jacobson algebra where $x \mapsto S_{-1}$, $y \mapsto S_1$. Since $\ind(S_{-1}) = 1 \neq 0$ the extension cannot be trivial.
\end{proof}

\begin{rmk}
Fredholm elements have previously been defined in \cite{liftingunits} by Perera and explored in the context of the algebraic $K$-theory of rings to identify obstructions to unit lifting modulo exchange ideals. Specifically, given an exchange ideal $I$ of a ring $R$, the set of Fredholm elements is defined as above, i.e. $\pi^{-1}(GL(R/I))$ where $GL(R/I)$ denotes the set of invertible elements in $R/I$. Then the extension $0 \to I \to R \to R/I \to 0$ gives rise to an exact sequence of $K$-groups
\[
\begin{tikzcd}
K_1(R) \arrow{r} &K_1(R/I) \arrow{r}{\delta} &K_0(I) \arrow{r} &K_0(R) \arrow{r} &K_0(R/I)
\end{tikzcd}
\]
where $\delta$ is a connecting map between the two $K$-groups (see \cite{rosenberg} for more information on the specific construction of $\delta$). The index is then defined as $\text{index}(x) = \delta(\pi(x))$ for any Fredholm element $x$; futhermore, this index map retains many of the properties of the index from the present article. This led Perera to a unit-lifting result which is very similar in style to Corollary \ref{conjugation}(4) above. While Perera's approach is more generally applicable ($M_\infty(K)$ is an exchange ideal of $\rcfm K$), the approach in this article uses linear algebraic techniques afforded by the algebra structure to show the unit-lifting property. These techniques then lend themselves towards explicit constructions of algebras, as we will see in the next section.
\end{rmk}

\section{A Menagerie of Extensions} \label{menagerie}
In this section, we will use the language of extensions to produce an infinite family of non-isomorphic matrix algebras which have faithful ideal $M_\infty(K)$ and quotient isomorphic to $K[x,x^{-1}]$.

\begin{ex}
	Fix the upper left corner embedding of $M_\infty(K)$ into $\rcfm K$, and then for each $n > 0$ define the family of maps $\psi_n: K[x,x^{-1}] \rightarrow Q(K)$ by $\psi_n(x) = S_{-n}$ and $\psi_n(x^{-1}) = S_n$. The pullback along $\psi_n$ and $\pi$ gives algebras of the form 
	
	\[\mathcal T_n = M_\infty(K) + \text{Span}_K\{S_{in} : i \in \Z\}.\]
	The case when $n = 1$ is the Jacobson embedding and is isomorphic to the Toeplitz-Jacobson algebra, $\mathcal T$. 
\end{ex}
\begin{claim}
	The algebras $T_n$ are pairwise non-isomorphic.
\end{claim}
\begin{proof}
	Suppose there was an isomorphism $\phi$ between the algebras $T_n$ and $T_m$; without loss of generality choose the isomorphism so that $n < m$. Consider $\phi(S_{-n})$. Because $S_{-n}$ is Fredholm, its image under $\phi$ must also be Fredholm. Furthermore, $\pi(\mathcal T_n) \simeq K[x,x^{-1}]$, and the only invertible elements of $K[x,x^{-1}]$ are the monomials, $\phi(S_{-n}) = S_{im} + k$ for some $i \in \Z$ and $k \in M_\infty(K)$. Since the dimension of the kernel and cokernel of a linear transformation are preserved under isomorphism, the index of a matrix must also be preserved. Because the index is invariant under perturbation by an element $M_\infty(K)$, the following equality holds: 
	\[n = \ind(\phi(S_{-n})) = \ind(S_{im} + k) = -im.\] Thus $n = m \cdot (-i)$, and it follows that $n = m$ since $n < m$.
\end{proof}

\begin{ex}
	Using Construction 2, and the upper left corner embedding of $M_\infty(K)$ into $\rcfm K$ we will construct a trivial extension of $K[x,x^{-1}]$ by $M_\infty(K)$. Define the map $\psi_0: K[x,x^{-1}] \rightarrow Q(K)$ by \[\psi_0(x^n) = \overline{D_2^n} := \overline{\diag(1, 2, \ldots, 2^{i-1}\ldots)}^n = \overline{\diag(1, 2^n, \ldots, 2^{n(i-1)}, \ldots)}\] for any $n \in \Z$. Extend linearly to define a map into $K[x,x^{-1}]$. Note that $\ind(\psi_0(x)) = 0$.
	
	\begin{claim}
		The set $D = \{D_2^n : n \in \Z\}$ is linearly independent modulo $M_\infty(K)$
	\end{claim}
	\begin{proof}
		Suppose that there is a finite set of $m$ nonzero coefficients $\{k_{n_i} : 1 \leq i \leq m \text{ and } n_i \in \Z\}$ where $n_1 < n_1 < \cdots < n_m$ and $\sum_{i} k_{n_i} D_2^{n_i} = M$ for some matrix $M \in M_\infty(K)$. Choose $j \in \zplus$ such that $M$ may be partitioned as follows:
		\[\left(\begin{array}{c | c}
			M' &0\\ \hline
			0 &0
		\end{array}\right)\] where $M' \in M_j(K)$. Ignoring the first $j$ rows of the matrix equation, one has the infinite system of $m$ linear equations over $K$
		\[\begin{array}{c c l}
			\sum_i k_{n_i} 2^{n(j)} &= &0\\
			\sum_i k_{n_i} 2^{n(j+1)} &= &0\\
			\sum_i k_{n_i} 2^{n(j+2)} &= &0\\
			\vdots\\
			\sum_i k_{n_i} 2^{n(j+m-1)} &= &0
		\end{array}\] 
		This then translates to a matrix equation $A v = 0$ where
		\[A = \begin{pmatrix}
			2^{n_1(j)} &2^{n_2 (j)} &\cdots &2^{n_m (j)}\\
			2^{n_1 (j+1)} &2^{n_2 (j+1)} &\cdots &2^{n_m(j+1)}\\
			\vdots &&&\vdots\\
			2^{n_1(j+m - 1)} &2^{n_2(j+m - 1)} &\cdots &2^{n_m(j+m - 1)}
		\end{pmatrix} \text{ and } v = \begin{pmatrix}
			k_{n_1} \\ k_{n_2}\\ \vdots \\k_{n_m}
		\end{pmatrix}\]
		The matrix $A$ is clearly Vandermonde, thus invertible. So $k_{n_i} = 0$ for all $1 \leq i \leq m$. 
	\end{proof}
	
	Because these images of $\psi_0$ are linearly independent in $Q(K)$ one may use Construction 2 to take the pullback of $Q(K)$ along $\pi$ and $\psi_0$ to form an extension $\mathcal T_0$ of $K[x,x^{-1}]$ by $M_\infty(K)$. The linear independence of $\{D_2^n : n \in \Z\}$ modulo $Q(K)$ certainly assures that $\psi$ is injective. Then as a consequence of Lemma \ref{faithful embed} and Lemma \ref{faithfulideal} $T_0$ may be considered to be a subalgebra of $\rcfm K$, and in addition $M_\infty(K)$ is a faithful ideal of $\mathcal T_0$. The extension $\mathcal T_0$ can then be seen to be of the form
	\[\mathcal{T}_0 = \text{Span}_K\left\{D_2^n : n \in \Z \right\} + M_\infty(K) \subseteq \rcfm K.\] 
	
	As noted previously, this extension $\mathcal T_0$ occupies a separate equivalence class of extensions than those from the previous example. In fact, this extension is not even isomorphic to the previous extensions. One can check this by examining of the the multiplication action of $\psi_0(x^n)$ on the set of matrix units $\{e_{ij} : i, j \in \zplus\}$. Multiplication merely scales the matrix units instead of shifting them as in the previous example.
\end{ex}

In the  previous two examples we have constructed an infinite family of extensions $\mathfrak T = \{\mathcal T_i : i \geq 0\}$ of $K[x,x^{-1}]$ by $M_\infty(K)$; the subscript of each extension corresponds to the index of the image of $x$ under $\psi_i$. We pose the following question:
\begin{question}
	Is $\mathfrak T$ a complete (up to equivalence) family of faithful extensions of $K[x,x^{-1}]$ by $M_\infty(K)$?
\end{question}

It is evident from our work in Section \ref{trivial} that two equivalent extensions must have the same index, but as of right now, the converse is unknown. A starting place would be a resolution of the following.

\begin{question}
	Are all trivial extensions of $K[x,x^{-1}]$ by $M_\infty(K)$ equivalent?
\end{question}

\section*{Acknowledgements}
I would like to thank Manfred Dugas and Daniel Herden for their close reading of many recent drafts of this article. Thanks also the the anonymous referee for their thorough treatment of this article.

\end{document}